\documentclass{elsarticle}
\usepackage{amsmath, amsthm, amssymb}

\usepackage{tikz}
\usepackage{pgfplots}
\usepackage{caption}
\usepackage{subcaption}

\definecolor{s1}{RGB}{228, 26, 28}
\definecolor{s2}{RGB}{55, 126, 184}
\definecolor{s3}{RGB}{77,175,74}
\definecolor{s4}{RGB}{152,78,163}
\definecolor{s5}{RGB}{255,127,0}
\definecolor{s6}{RGB}{166,86,40}

\newtheorem{theorem}{Theorem}
\newtheorem{lemma}{Lemma}

\newtheorem{proposition}{Proposition}
\newtheorem{corollary}{Corollary}

\DeclareMathOperator{\interior}{int} 
\DeclareMathOperator{\tr}{tr} 
\DeclareMathOperator{\Span}{span}
\DeclareMathOperator*{\argmin}{arg\,min}

\def\ESS{\mathcal E}
\def\NSS{\mathcal V}
\newcommand{\spa}{\mathcal{X}}
\newcommand{\laplace }{\Delta}
\newcommand{\projL}{\Pi_L}
\newcommand{\jordan}{\circ}

\newcommand{\E}[1]{E_{#1}}
\newcommand{\V}[1]{V_{#1}}
\newcommand{\dd}{\,\mathrm d}
\newcommand{\Sym}[1]{\mathcal S^{#1}}
\newcommand{\LowerCorner}{\Sym{p}_{\ulcorner}}
\newcommand{\gl}[1]{\mathbb{GL}^{#1}}
\newcommand{\cd}{\,|\,}
\newcommand{\transm}{^{\top}}
\newcommand{\transp}{^*}
\newcommand{\ipd}[2]{\left\langle #1, #2 \right\rangle_d}
\newcommand{\ipp}[2]{\left\langle #1, #2 \right\rangle_p}
\newcommand{\normd}[1]{\left\| #1 \right\|_d}
\newcommand{\normp}[1]{\left\| #1 \right\|_p}

\makeatletter
 \def\imod#1{\allowbreak\mkern10mu({\operator@font mod}\,\,#1)}
 \makeatother

\begin{document}

\begin{frontmatter}
\title{Linear Estimating Equations for Exponential Families with Application to Gaussian Linear Concentration Models} 
\author{Peter G. M. Forbes and Steffen Lauritzen}
\address{Department of Statistics, University of Oxford \\ 1 South Parks Road, Oxford OX1 3TG, United Kingdom}
\begin{abstract}
In many families of distributions, maximum likelihood estimation is intractable because the normalization constant for the density which enters into the likelihood function is not easily available.  The score matching estimator \citep{hyvarinen:05} provides an alternative  where this normalization constant is not required.  The corresponding  estimating equations become linear for an exponential family \citep{gidas,Hyvarinen2}. The score matching estimator is shown to be consistent and asymptotically normally distributed for such models, although not necessarily efficient. Gaussian linear concentration models are examples of such families. For linear concentration models that are also linear in the covariance \citep{jensen:88} we show that the score matching estimator is identical to the maximum likelihood estimator, hence in such cases it is also efficient. Gaussian graphical models  and graphical models with symmetries \citep{Hojsgaard} form particularly interesting subclasses of linear concentration models and we investigate  the potential use of the score matching estimator for this case. \\

\end{abstract}

\end{frontmatter}

\section{Introduction}

The increasing interest in analysis of high-dimensional data has necessitated the development of parsimonious multivariate models with reliable, computationally efficient estimation procedures.  For example, sparse Gaussian graphical models \citep{dobra:etal:04,ma:etal:07,rothman:etal:08,bickel:levina:08,chandrasekaran:etal:12} have drawn significant interest and several computationally efficient estimation procedures have been developed \citep{banerjee:etal:06,banerjee:etal:08,friedman:etal:08}.  Gaussian graphical models with symmetry \citep{Hojsgaard} form another example, though no efficient estimation procedures have yet been developed. Here we describe and exploit a method which provides linear estimating equations when applied to any exponential family, in particular to any Gaussian graphical model with symmetry.  In contrast to the maximum likelihood estimator, which often requires iterative methods, this \emph{score matching estimator} is computationally efficient for such families and therefore has potential for initial model screening.  Even when the maximum likelihood estimator is desired, it must be computed iteratively and the score matching estimator may provide a useful initial value for the iterations.

\section{Preliminaries}

Consider a random quantity taking values in an open subset $\spa$ of $\mathbb{R}^p$ which we consider equipped with the standard inner product $\ipp{\cdot}{\cdot}$, the associated norm $\normp{\cdot}$,  and the canonical basis. 
 
Throughout the paper,  $\mathcal P$ denotes a class of distributions over $\spa$ with twice continuously differentiable densities with respect to the Lebesgue measure on $\spa$.  The general developments below are equally valid for $\spa$ being a Riemannian manifold with associated geometric measure \citep{dawid:lauritzen:05}, but as our main focus is the multivariate Gaussian distribution we shall refrain from working at this level of generality.
We use $\nabla$ for the gradient and $\laplace$ for the Laplace operator on $\spa$ so that
\[\nabla f(x) = \left\{\frac{\partial}{\partial x_i} f(x)\right\} \in \mathbb{R}^p, \quad \laplace f(x) = \sum_{i=1}^p \frac{\partial^2}{\partial x_i^2}f(x).\]  

\subsection{Scoring rules}

A \emph{scoring rule} $S(x,Q)$ is a real-valued function which quantifies the accuracy of a predictive distribution $Q\in\mathcal P$ upon observing the realized value $x\in\spa$.  It is (strictly) \emph{proper} if the expected value $\E{X\sim P} S(X,Q)$ is (uniquely) minimized over $\mathcal P$ at $Q=P$.  Two scoring rules are \emph{equivalent} if they differ by a positive scalar multiple and a function of $x$ alone.

Every proper scoring rule induces a \emph{divergence} \citep{dawid:98,grunwald:dawid:04}: \[d(P,Q)=\E{X\sim P}\{S(X,Q) - S(X,P)\}.\]  The divergences of equivalent scoring rules are  proportional. A much used scoring rule is the \emph{log score} $S(x,Q)=-\log g(x)$, where $g$ is the density of $Q$ \citep{bernardo:79,good:52,mccarthy};  the corresponding divergence is then the Kullback--Leibler divergence.

Given an independent sample $x^1,\ldots, x^n$ with empirical distribution $\hat P$ and unknown distribution $Q$, the \emph{optimal score estimator} \citep{gneiting}  $\hat Q$ of $Q$ is determined as the minimizer of the  empirical score 
\[\hat Q=\argmin_{Q\in\mathcal P} \E{X\sim \hat P}\{S(X,Q)\}=\argmin_{Q\in\mathcal P} \sum_{i=1}^nS(x^i, Q)=\argmin_{Q\in\mathcal P} d(\hat P,Q) . \]
The first two expressions are well-defined even when $\hat P\not\in \mathcal{P}$ whereas the latter may not be.  \citet{dawid:lauritzen:05} show that for a parametrized family $\mathcal{P}=\{Q_\theta: \theta\in \Theta\}$ with $\Theta$ being an open subset of $\mathbb{R}^d$, this minimization gives rise to an \emph{unbiased estimating equation} \citep{godambe}
\begin{equation}\label{eq:estimating}
\sum_{i=1}^nS'(x^i, \theta)=0,
\end{equation}
where $S'(x,\theta)$ is the vector of derivatives of $S(x,Q_\theta)$ w.r.t.\ $\theta$.
Solutions to such equations are also known as M-estimators \citep{huber:64,huber:67} and these are typically consistent and asymptotically normal although not necessarily efficient. If $S(x,Q)=-\log g(x)$ is the log score, the equation \eqref{eq:estimating} is the likelihood equation and the corresponding estimator is the maximum likelihood estimator (MLE).

\subsection{Score matching estimator}

Suppose the density $g$ of $Q\in \mathcal{P}$ is twice continuously differentiable and satisfies the regularity assumptions:
\begin{equation}\label{eq:finitereg}
 \E{X\sim P}{\normp{\nabla \log g(X)}^2}<\infty \text{ for all $P,Q\in \mathcal{P}$}
 \end{equation} as well as
 \begin{equation}\label{eq:boundary}
 \text{$g(x)$ and $\normp{\nabla g(x)}$ tend to zero as $x$ approaches the boundary of $\spa$.}
\end{equation}
Then, using integration by parts, \citet{hyvarinen:05,Hyvarinen2} showed that the divergence function
\begin{equation*}
d_2(P,Q)=\E{X\sim P} \normp{\nabla\log g(x) - \nabla\log f(x)}^2
\end{equation*}
where $f$ is the density of $P$,
is induced by the scoring rule
\begin{equation*}
S_2(x,Q) =\frac{1}{2}\normp{\nabla\log g(x)}^2+\laplace \log g(x).
\end{equation*} This scoring rule can be shown to be proper \citep{dawid:07,dawid:lauritzen:05}.  % and it is clearly 2-local and homogeneous \citep{ehm:gneiting:12,parry}.
The \emph{score matching estimator} (SME) is the optimal score estimator for this scoring rule. Note that the SME is not invariant under transformations of $x$, nor under change of base measure.  Hence care must be taken when choosing the representation of the data to ensure the resulting estimator is suitable.

\section{Exponential families}

As indicated in \citep{Hyvarinen2}, the SME is particularly simple when $\mathcal P$ is an exponential family with densities $g(x\cd \theta)$ where
\begin{equation}
\log g(x \cd \theta) = \ipd{\theta}{t(x)} - a(\theta) +b(x),\quad \theta\in\Theta.
\label{eq:expFamily}
\end{equation}
Here $t(x)\in L$ is the canonical sufficient statistic, $L$ is a $d$-dimensional vector space, $\ipd{\cdot}{\cdot}$ an inner product on $L$, and $\Theta\subseteq L$ is the (convex) canonical parameter space
\[\Theta=\interior \left\{\theta\in L: a(\theta)=\log 
\int_\spa e^{\ipd{\theta}{t(x)}} \dd x<  \infty\right\}.\] 
Let $D(x)$ be the linear map from $L$ to $\mathbb{R}^p$ determined by the equation
$D(x)\eta = \nabla \ipd{\eta}{t(x)}$ for all $\eta\in L$.  
  Then we have
\begin{equation}\label{eq:gradient}\nabla\log g(x\cd \theta) = D(x)\theta+ \nabla b(x).\end{equation}  
Further we get 
\[\laplace \log g(x\cd \theta) = \ipd{\theta}{\laplace t(x)} + \laplace b(x)\] where, similarly,
 $\laplace t$ is given by $\ipd{\eta}{\laplace t(x)}=\laplace \ipd{\eta}{t(x)}$. 
We assume that the regularity conditions \eqref{eq:finitereg} and \eqref{eq:boundary} are satisfied; thus in particular both $\E{\theta}\normp{D(X) \eta}^2$ and $\E{\theta}\normp{\nabla b(X)}^2 $ are finite. 

Let $D(x)\transp$ be the transpose of $D(x)$ given by $\ipp{y}{D(x)\eta}=\ipd{D(x)\transp y}{\eta}$ for all $\eta\in L$ and $y\in \mathbb{R}^p$. 
We furthermore assume that the linear map $\Psi(\theta)=\E{\theta}\{D(X)\transp D(X)\}$ from $L$ to $L$ is invertible so the corresponding quadratic form \[
\ipd{\eta}{\Psi(\theta) \eta}= \E{\theta}\normp{D(X)\eta}^2\] is positive definite, i.e.\ it is non-zero unless $\eta=0$. 
The objective function $J_2(\theta)=\sum_{i=1}^nS_2(x^i,Q_\theta)$ becomes
\[
\sum_{i=1}^n\frac{1}{2}\normp{D(x^i)\theta}^2+\ipd{\theta}{D(x^i)\transp\nabla b(x^i)}+\ipd{\theta}{\laplace  t(x^i)}+\frac{1}{2}\normp{\nabla b(x^i)}^2+\laplace  b(x^i).
\]
  The last two terms depend on $x$ only and will henceforth be ignored: this yields an equivalent scoring rule and does not alter the SME.
	
The objective function $J_2$ depends quadratically on $\theta$ and the minimizer of $J_2$ is unique if and only if the quadratic form on $L$
\begin{equation}
\label{eq:D2}
D_2(\eta)=\sum_{i=1}^n \normp{D(x^i)\eta}^2
\end{equation}
is positive definite, i.e.\ if $D(x^i)\eta= 0$ for $i=1,\ldots,n$ implies $\eta=0$.
The SME is the minimizer of $J_2(\theta)$ over $\Theta$, leading to the  linear estimating equation for $\theta$ 
\begin{equation}
J'_2(\theta)=\sum_{i=1}^n  D(x^i) \transp \left\{ D(x^i)\theta +\nabla b(x^i) \right\} +\laplace  t(x^i) =0.
\label{eq:score}
\end{equation}
Thus, provided $\sum_{i=1}^n D(x^i)\transp D(x^i)$ is invertible, the score estimation equation has the unique solution 
\begin{equation}
\label{eq:scorematchingequation}
\check\theta_n = -\left\{\sum_{i=1}^n D(x^i)\transp D(x^i) \right\}^{-1} \sum_{i=1}^n  \left\{D(x^i)\transp \nabla b(x^i) +\laplace  t(x^i)\right\}.
\end{equation}
In the case $b(x)=0$ this is equivalent to the \emph{variational estimator} (eq. 1.9) of \citet{gidas} and to eq. 34 of \citep{Hyvarinen2}.  By taking the inner product of \eqref{eq:score} with the minimizer $\check \theta_n$ we further obtain
\[
\sum_{i=1}^n \normp{D(x^i) \check\theta_n}^2=-\sum_{i=1}^n\left\{\ipd{\check\theta_n}{D(x^i)\transp\nabla b(x^i)}+\ipd{\check\theta_n}{\laplace  t(x^i)}\right\}.
\]
Inserting this relation into the expression for $J_2(\check\theta_n)$ yields a linear dependence of the minimal value of $J_2$ on $\check\theta_n$ with the simple expression
\begin{equation}
J_2(\check\theta_n)=\sum_{i=1}^n \ipd{\check\theta_n}{D(x^i)\transp\nabla b(x^i)+\laplace  t(x^i)}/2%+\normp{\nabla b(x^i)}^2+2\laplace  b(x^i).
 \label{eq:optscore}
\end{equation}
where we have ignored the terms depending on $x$ alone.  

As mentioned earlier, SMEs are M-estimators and thus typically consistent, as also claimed in Corollary 3 of \citep{hyvarinen:05}. However, the argument in \citep{hyvarinen:05} is incomplete since the convergence of $J_2/n$ to its expectation does not  imply that the minimizer $\argmin_{\phi} J_2(\phi)$ converges to $\theta$ without additional conditions on $J_2$.  General consistency results are established under suitable regularity conditions in \citep{huber:67}, but for an exponential family we can establish consistency directly.

\begin{proposition}Assume $\mathcal P= P_\theta, \theta \in \Theta$ is an exponential family as above and $X^1,\ldots, X^n$ is a sample of size $n$ from $P_\theta$. Then the SME $\check\theta_n$ is asymptotically consistent for $\theta$. 
\label{prop:consistency}
\end{proposition}
\begin{proof}
Each term in the last sum of \eqref{eq:scorematchingequation} has expectation
\[
\E{\theta}\{D(X) \transp \nabla b(X) +\laplace  t(X)\}=\int g(x\cd \theta)D(x)\transp \nabla b(x)\dd x + \int g(x \cd \theta)\laplace t(x)\dd x
\]
and both integrals on the right-hand side are finite by our assumptions.
Using integration by parts on the second term yields \[\int g(x\cd \theta)\laplace t(x)\dd x = -\int D(x)\transp \nabla g(x\cd \theta)\dd x,\] since the boundary terms vanish by \eqref{eq:boundary}.  
Using that $\nabla  g(x)=g(x)\nabla \log g(x)$ and substituting the expression \eqref{eq:gradient} for  $\log g(x\cd \theta)$ on the right-hand side  yields \[
\int g(x \cd \theta)\laplace t(x)\dd x
=-\int g(x \cd \theta)D(x)\transp \{D(x)\theta+ \nabla b(x)\}
\dd x,\] which then gives
\[
\E{\theta}\{D(X) \transp \nabla b(X) +\laplace  t(X)\}
= -\E{\theta}\{D(X)\transp D(X) \}\theta=-\Psi(\theta)\theta.\]
Thus, by the law of large numbers
\[\lim_{n\to \infty}
\frac{1}{n}\sum_{i=1}^n  \left\{D(X^i)\transp\nabla b(X^i) +\laplace  t(X^i)\right\} = -\Psi(\theta)\theta
\]in probability as well as with probability one. Similarly, for the first sum in \eqref{eq:scorematchingequation} we get
\[\lim_{n\to \infty}
\frac{1}{n}\sum_{i=1}^n  \left\{D(X^i)\transp D(X^i) \right\} = \Psi(\theta)
\]
which is invertible by assumption.  Hence the estimate $\check\theta_n $ converges to $\theta$ and is thus  consistent for $\theta$, as desired.
\end{proof}

Under  a few additional assumptions, we can also show that the SME is asymptotically normally distributed. \begin{proposition} \label{prop:normality}Assume $\mathcal P= P_\theta, \theta \in \Theta$ is an exponential family as above and $X^1,\ldots, X^n$ is a sample of size $n$ from $P_\theta$. Assume further that all of $
\E{\theta}\normd{D(X)\transp \nabla b(X)}^2$, $ \E{\theta}\normd{\laplace t(X)}^2$, and $ \E{\theta}\normd{D(X)\transp D(X) \theta}^2$ are finite. Then  $\sqrt{n}(\check\theta_n-\theta)$ converges in distribution to a normal distribution on $L$ with mean zero. 
\end{proposition}
\begin{proof}
From \eqref{eq:scorematchingequation} we get
\begin{equation}
\sqrt{n}(\check\theta_n-\theta) = 
-\Psi_n^{-1} \sum_{i=1}^n  \left\{D(X^i) \transp \nabla b(X^i) +\laplace  t(X^i)+D(X^i)\transp D(X^i)\theta\right\}/\sqrt{n} \label{eq:standard}\end{equation}
where
\[\Psi_n=\frac{1}{n}\sum_{i=1}^n D(X^i)\transp D(X^i).\]
As  in the proof of Proposition~\ref{prop:consistency},  we conclude that $\Psi_n$ converges in probability to its expectation $\Psi(\theta)$, which is invertible by assumption.  The $i^{th}$ term of the  sum in \eqref{eq:standard} has expectation zero and finite variance and hence,  by the Central Limit Theorem, the normalized sum converges in distribution to a normal distribution on $L$ with mean zero. By Slutsky's theorem, so does  $\sqrt{n}(\check\theta_n-\theta)$, as desired. 
\end{proof}

The asymptotic inverse covariance of the SME is $nG(\theta)$ where $G(\theta)$ is the \emph{Godambe information}  \citep{godambe} 
\[ G(\theta)=\Psi(\theta)H(\theta)^{-1}\Psi(\theta).\]
Here \[H(\theta)= \V{\theta}\left\{D(X)\transp D(X)\theta +D(X)\transp \nabla b(X)  +\laplace  t(X)\right\};\]
see for example \cite[Section 9.2]{barndorff:cox:94} or \cite[Theorem 5.21]{vaart:98}. Note that, as we have shown above, $\Psi_n$ is a consistent estimator of $\Psi(\theta)$ and $H(\theta)$ can be estimated consistently by the corresponding empirical covariance. 

For finite $n$ it is possible that $\check\theta\not\in\Theta$.  Even in this case $\check\theta$ itself may be useful: the value of $J_2(\check\theta)$ can be very quickly computed and used for  model screening,  or $\check\theta$ can be used as a starting value for iterative estimation methods.  Consistency ensures that $\check\theta\in\Theta$ for $n$ sufficiently large.

As noted earlier, the score matching equation is not invariant under data reparametrization, nor under change of base measure.  We could in principle use the estimating equation after reducing to the sufficient statistic $t=\sum_i t(x^i)$,  which has density $\tilde g(t\cd \theta)$ where 
\[
\log \tilde g(t\cd\theta)= \ipd{\theta}{t}-na(\theta)+ h_n(t)
\]
and where $\exp\{h_n(t)\}$ is the density of the product of measures $\exp\{b(x^i\}\dd x^i$ transformed by the sufficient statistic.  Then the SME becomes $\check\theta=-\nabla h_n(t)/n$, which coincides with both \citeauthor{martinlof}'s exact estimator \citep{martinlof} and the maximum plausibility estimator \citep{maxplausibility} for this case.  The exact estimator is known to be consistent and efficient \citep{hoglund:74}.  Unfortunately, the form of $h_n(t)$ is often intractable and the exact estimator is often more difficult to calculate than the MLE.  

We have chosen not to reduce by sufficiency: our SME may lose statistical efficiency, but it gains computational speed from the simplicity of its estimating equations.

\section{Gaussian linear concentration models}

We now consider Gaussian models with linear structure in the concentration matrix \citep{anderson:70}, exploiting that they are special instances of the exponential families discussed above.  

Let $L$ be a $d$-dimensional subspace of $\Sym{p}$, the symmetric $p\times p$ matrices equipped with the trace inner product $\ipd{A}{B}= \tr(AB)$ and associated \emph{Frobenius} norm $\normd{A}^2=\tr(A^2)$.  Consider the family of Gaussian densities
\begin{align*}
\log p(x\cd K)& =  \{\log\det(K)-p\log(2\pi)-\ipp{x}{Kx}\}/2
\\&=-\ipd{K}{xx\transm}/2+\{\log\det(K)-p\log(2\pi)\}/2
\\&=\ipd{K}{-\projL(xx\transm)/2}+\{\log\det(K)-p\log(2\pi)\}/2
\end{align*}
which clearly has the form 
\eqref{eq:expFamily} with $\spa=\mathbb R^p$, $\theta=K$, $a(K)=p\log(2\pi)/2-\log\det(K)/2$, $b(x)=0$, and $t(x)= -\projL(xx\transm)/2$, where $\projL$ is the orthogonal projection onto $L$ in $\Sym{p}$. The canonical parameter space is $\Theta= L\cap \Sym{p}_+$, where $\Sym{p}_+$ is the set of positive definite symmetric $p\times p$ matrices, \citep[p.\ 116]{olesbog}.

The maps discussed above become
\begin{equation}\label{eq:basicmaps}
D(x)K=-Kx,\quad D(x)\transp y = -\projL(xy\transm +yx\transm)/2,\quad \laplace t(x)=-\projL(I_p),\end{equation}
where $I_p$ is the $p\times p$ identity matrix.  For simplicity we assume in the following that $I_p\in \Theta$; this can always be achieved for non-empty $\Theta$ by choosing a suitable basis for $\mathbb R^p$.
Then we have $\projL(I_p)=I_p$ so the Laplacian becomes $\laplace t(x)=-I_p$.

To see that  \eqref{eq:basicmaps} holds, note that for any $K\in L$ we have 
\[\nabla\ipd{K}{\projL(xx\transm)}=\nabla\ipd{K}{xx\transm}= \nabla \tr(Kxx\transm)
= 2Kx\]
so $D(x)K=-Kx$. Furthermore we have
\begin{align*}
\ipd{D(x)\transp y}{K}& = \ipp{ y}{D(x)K} = -\ipp{y}{Kx}=-\tr(xy\transm K)
\\&= -\ipd{xy\transm +yx\transm}{K}/2=\ipd{-\projL(xy\transm +yx\transm)/2}{K}
 \end{align*}
and finally
\[\laplace \ipd{K}{\projL(xx\transm)}=\laplace\ipd{K}{xx\transm}
=2\tr(K)=2\ipd{K}{I_p}.\]
In particular we get
\begin{equation}
\label{eq:DDK}
D(x) \transp D(x)  K = -D(x)\transp Kx =\projL(xx\transm K+ K xx\transm)/2=\projL (K\jordan xx\transm),
\end{equation}
where $A\jordan B= (AB\transm + BA\transm)/2$ is the \emph{Jordan product} \citep{albert:46} of the symmetric matrices $A$ and $B$. 
The estimating equation \eqref{eq:score} now specializes to
\begin{equation}\label{eq:smeconc}\projL (K\jordan W)= I_p,\end{equation}
where we have let $W=n^{-1}\sum_{i=1}^n x^i {x^i}\transm$ be the scaled Wishart matrix of sums of squares and products.
The expression \eqref{eq:optscore} for $J_2$ becomes simply $-n\tr{\check K/2}$ which can be evaluated very quickly. 

We next verify that $p(x\cd K)$ satisfies the assumptions for consistency and asymptotic normality.  For consistency we must satisfy the regularity conditions \eqref{eq:finitereg} and \eqref{eq:boundary}.  It is obvious that both $p(x\cd K)$ and $\normp{\nabla p(x\cd K)}=|p(x\cd K)|\normp{Kx}$ tend to zero as $\normp{x}\rightarrow\infty$, and furthermore
\begin{align*}
\E{\theta}\normp{\nabla\log p(X\cd K)}^2 =\E{\theta}\{\tr(KXX\transm K)\} = \tr(K\theta^{-1}K)<\infty.
\end{align*}

For asymptotic normality we note that  $\E{\theta}\normd{D(X) \nabla b(X)}^2=0$ since $b(x)=0$.  Furthermore $\E{\theta}\normd{\laplace t(X)}^2=\E{\theta}\{\tr(I_p)\}=p<\infty$.  Finally,
\begin{align*}
\E{\theta}\{\normd{D(X)\transp  D(X) K}^2\}&=\E{\theta} [\{\tr\projL (K\jordan XX\transm)\}^2]\\
&\leq \E{\theta}[ \{\tr(K\jordan XX\transm)\}^2]\\
&=\E{\theta}\{(X\transm K X)^2\},
\end{align*}
which is the fourth moment of a normal distribution and hence finite.  Thus all the conditions for consistency and asymptotic normality are satisfied.  We refrain from calculating a specific expression for the Godambe information. 

\subsection{Jordan linear concentration models}

Consider the special case where $L$ is closed under the Jordan product, i.e.\ it forms a Jordan subalgebra of $\Sym{p}$, or equivalently $\Theta = L\cap \Sym{p}_+$ is closed under inversion. Such hypotheses are exactly those which are linear in both the covariance and inverse covariance \citep{jensen:88}.  In particular, $L$ contains all models which are determined by invariance under a subgroup of the general linear group \citep{andersson:75}. We shall show that the MLE and the SME coincide for these models. First we need a lemma.

\begin{lemma}\label{lem:jordanproj}
If $L$ is a Jordan subalgebra of $\Sym{p}$ then for any $A\in L$ and $B\in \Sym{p}$ we have $\projL(A\jordan B)= A\jordan \projL(B)$.
\end{lemma}
\begin{proof}
Let $B_0 = \projL(B)$. Clearly we have $A\jordan B_0\in L$ since $L$ is closed under the Jordan product. Further, for any $K\in L$ we have
\begin{align*}
\ipd{A\jordan B-A\jordan B_0}{K}&=\tr(ABK)-\tr(AB_0K)
\\&=%\tr(BKA)-\tr(B_0KA)=
\ipd{B-B_0}{K\jordan A}=0
\end{align*}
where the last equality follows because $K\jordan A \in L$. Thus $A\jordan B-A\jordan B_0$ is orthogonal to $L$ and the lemma follows.
\end{proof}
\noindent We now obtain the desired result.
\begin{theorem}If the subspace $L$ is a Jordan subalgebra then the SME is equal to the MLE. Furthermore, if $\projL(W)$ is invertible we have
\[\hat K = \check K = \{\projL(W)\}^{-1}.\]
\end{theorem}
\begin{proof} 
The family is a full and canonical exponential family with $\projL(W)$ as the canonical sufficient statistic, and hence the likelihood equation becomes
\[\projL(W) = \E{K}\{\projL(W)\} = \projL(K^{-1})= K^{-1}.\]
This implies $\hat K = \{\projL(W)\}^{-1}$.

As $L$ is a Jordan subalgebra we have $I_p\in L$. Using Lemma~\ref{lem:jordanproj}, the score matching equation \eqref{eq:smeconc} reduces to
\[\projL(K\jordan W)=K\jordan \projL(W)=I_p,\]
whence we get $\check K = \{\projL(W)\}^{-1}$. This completes the proof.
\end{proof}

\subsection{Existence and uniqueness}

Having observed a sample $x=(x^1,\ldots,x^n)$, the score matching equation \eqref{eq:smeconc} has a unique solution if and only if any of the following equivalent conditions hold:
\begin{enumerate}
\item The quadratic form $D_2(K)=\sum_{i=1}^n \|Kx^i\|^2$ is positive definite on $L$
\item $K=0$ is the only element of $L$ which maps all $x^i$ to zero
\item The kernel of the linear map $K\to \projL(K\jordan W)$ is trivial.
\end{enumerate}
We say \emph{the SME exists} if \eqref{eq:smeconc} has a unique solution, ignoring the fact that $\check K$ may not be positive definite. We have the following relation between existence of the SME and the MLE. 

\begin{proposition}\label{prop:sme2mle} Consider a Gaussian linear concentration model and let $W$ be an empirical covariance matrix as above. If the SME for $W$ exists, then the MLE for $W$ also exists.
\end{proposition}
\begin{proof} We proceed by assuming that the MLE for $W$ does not exist and showing that the SME does not exist either.
If the MLE does not exist, the convex level set of the likelihood function
\[C =\{ K\in \Theta =L\cap\Sym{p}_{+} : \ell(K)\geq \ell(I_p)\}\] must be unbounded. Here $\ell(K)$ is proportional to the logarithm of the likelihood function 
\[\ell(K) = \log\det (K) -\tr(KW)\] and we have assumed without loss of generality that $I_p\in L$. This implies \citep[Theorem 5.5]{olesbog} that there is an $A\in \Theta$ such that $I_p+ t A \in C$ for all $t\geq 0$, i.e.
\[\ell(I_p+t A) \geq \ell(I_p) \mbox{ for all $t\geq 0$,}\] or equivalently that 
\[\log \det (I_p+ t A) \geq t \tr(A W) \mbox{ for all $t\geq 0$.}\] But  if  $a_i$ are the positive eigenvalues of $A$ we have
\[\log \det (I_p+ t A) = \sum_{i=1}^p \log (1+ t a_i)\]
and for large $t$ this grows slower with $t$ than any line through the origin with positive slope. Thus we must have $\tr(AW)= \tr(A\jordan W)\leq 0$. Hence, since $A, W,$ and thus $A\jordan W$ are positive semidefinite, we have $A\jordan W=0$ and thus $\projL(A\jordan W)=0$. Thus the third condition for the existence of the SME does not hold.
\end{proof}

By choosing an orthogonal basis $e^1, \ldots, e^d$ for $L$, we can write the the matrix for the quadratic form $D_2$ as $M(x)= \{m_{uv}(x)\}$ where
\begin{equation}
m_{uv}(x) = \sum_{i=1}^n \ipp{e^ux^i}{e^vx^i} = n\tr(e^u W e^v).
\label{eq:M}
\end{equation}
Hence $D_2$ is positive definite if and only if $ \det M(x) >0$. This determinant is a polynomial in $x$ and hence it either holds that $\det M(x)=0$ for all $x$ or $\det M(x)>0$ except for a set of Lebesgue measure zero \citep{Okamoto}. In other words, either the SME exists with probability one, or else it never exists. This is in contrast to the MLE, which can exist with some probability strictly between zero and one \citep{buhl,uhler}. 

We shall say that the linear space $L$ is \emph{$n$-estimable} if the SME exists with probability one, or equivalently, if there is an $x=(x^1,\ldots,x^n) \in \mathbb R^{p\times n}$ such that $\det M(x)>0$. 

For $n\geq p$ it is well-known that $W$ is positive definite with probability one and hence $M(x)$ is positive definite and any $L$ is $n$-estimable. We may thus without loss of generality assume $n<p$ in the following.  As many high-dimensional data sets have $n$ much less than $p$, this case is highly relevant. Let $r=p-n$ and $T_k=k(k+1)/2$ denote the $k^{th}$ triangular number.

\begin{proposition}Let $L$ be a linear subspace of $\Sym{p}$. If  $\dim L>T_p-T_{r}$ then $L$ is not $n$-estimable. 
\label{prop:notEstimable}
\end{proposition}
\begin{proof} Let $\mathbb{X}=\Span\{x^1,\ldots,x^n\}$ and let 
$\Sym{p}_0(\mathbb{X})=\{K\in\Sym{p} : \mathbb{X}\subseteq\ker(K)\}$
be the  space of symmetric matrices that send all vectors in $\mathbb{X}$ to zero.  Since $\dim(\mathbb{X})= n$ with probability one, we have  $\dim\{\Sym{p}_0(\mathbb{X})\}= T_r$. Noticing that the quadratic form $D_2$ is positive definite over $L$ if and only if $L\cap \Sym{p}_0(\mathbb{X})= \{0\}$, and that
\begin{equation*}
\dim\{L\cap \Sym{p}_0(\mathbb{X})\} =  \dim L + \dim\{\Sym{p}_0(\mathbb{X})\} - \dim\{\Sym{p}_0(\mathbb{X})+ L\}\geq d+T_r-T_p,
\end{equation*}
we see that $D_2$ is not positive definite if $T_r+d>T_p$.
\end{proof}

Unfortunately the converse is not always true and it may happen that $L$ is not $n$-estimable even if $\dim L \leq T_p-T_r$.  In particular we note that
\begin{equation}\label{eq:counterex}
L=\left\{a e^1+be^2+ce^3+fe^4=\begin{pmatrix}
a & c & 0 & f \\
c & b & -f & 0 \\
0 & -f & a & c \\
f & 0 & c & b
\end{pmatrix} : a,b,c,f \in\mathbb R\right\}
\end{equation}
yields  a counterexample. We have $p=4$ and $\dim L=4$ so if we consider a single observation, we have $T_p-T_r= T_4-T_3 =4 =\dim L  $.   Letting $x=(x_1,x_2,x_3,x_4)$ be our single observation, the corresponding quadratic form $m_{uv}(x)$ is
\[
\begin{pmatrix}
x_{{1}}^{2}+x_{{3}}^{2}&0&x_{{1}}x_{{2}}+x_{{3}}x_{{4}}&x_{{1}}x_{{4}}-x_{{2}}x_{{3}}\\
0&x_{{2}}^{2}+x_{{4}}^{2}&x_{{1}}x_{{2}}+x_{{3}}x_{{4}}&x_{{1}}x_{{4}}-x_{{2}}x_{{3}}\\
x_{{1}}x_{{2}}+x_{{3}}x_{{4}}&x_{{1}}x_{{2}}+x_{{3}}x_{{4}}&x_{{1}}^{2}+x_{{2}}^{2}+x_{{3}}^{2}+x_{{4}}^{2}&0\\
x_{{1}}x_{{4}}-x_{{2}}x_{{3}}&x_{{1}}x_{{4}}-x_{{2}}x_{{3}}&0&x_{{1}}^{2}+x_{{2}}^{2}+x_{{3}}^{2}+x_{{4}}^{2}
\end{pmatrix}.
\]
Direct computation shows that $m_{uv}(x)$ is singular since it has the zero eigenvector
\[
\begin{pmatrix}
x_2^2+x_4^2, & x_1^2+x_3^2, & -x_1x_2-x_3x_4, & x_2x_3-x_1x_4 
\end{pmatrix}.
\]
This particular $L$, investigated by \citet{jensen:88}, is an example of a Jordan subalgebra and we  conclude  --- as \citeauthor{jensen:88} ---  that  the MLE also fails to exist.

Gaussian graphical models \citep{dempster:72} are special instances of linear concentration models. For example, in the model given by the four-cycle (Figure~\ref{fig:fourcycle}),
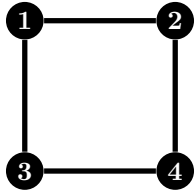
\begin{figure}[htb]
\center
\begin{tikzpicture}
[node/.style={circle,fill=black,inner sep=0pt,minimum size=5mm,text=white,font=\bfseries},
edge/.style={draw=black,line width=2pt}]
\node[node] (1)  at (-1,1) {1};
\node[node] (2) at (1,1) {2};
\node[node] (4) at (-1,-1) {3};
\node[node] (3) at (1,-1) {4};
\draw[edge] (1)--(2);
\draw[edge] (2)--(3);
\draw[edge] (3)--(4);
\draw[edge] (4)--(1);
\end{tikzpicture}
\caption{A four-cycle Gaussian graphical model.}
\label{fig:fourcycle}
\end{figure}
the MLE can only be calculated with iterative methods. If $n=2$, the MLE may or may not exist, whereas for $n=3$ the MLE always exists \citep{buhl}.  For $n=2$ we have $T_p-T_r=10-3 =7$, and since $\dim L = 8$, Proposition~\ref{prop:notEstimable} yields that the SME does not exist. For $n=3$ and above both the SME and MLE exist with probability one, the SME being a solution to a system of eight linear equations.

In  the following we list a number of facts about $n$-estimability which may assist in determining whether a given subspace $L$ is $n$-estimable. In particular,  we show that a subspace of an $n$-estimable space is $n$-estimable, and that a change of coordinate system does not affect $n$-estimability.
 
\begin{lemma}\label{lem:subspace} If $L$ is $n$-estimable and $L_0\subseteq L$, then $L_0$ is $n$-estimable. If $L$ is $n$-estimable with $n'>n$, then $L$ is $n'$-estimable.
\end{lemma}
\begin{proof}
The first statement follows since  $L_0\cap \Sym{p}_0(\mathbb{X})\subseteq L\cap \Sym{p}_0(\mathbb{X})= \{0\}$. If we let $\mathbb{Y}=\mbox{span}\{x^1,\dots,x^{n'}\}$ we have  $\Sym{p}_0(\mathbb{Y})\subseteq \Sym{p}_0(\mathbb{X})$ and the second statement follows since $L\cap \Sym{p}_0(\mathbb{Y})\subseteq L\cap \Sym{p}_0(\mathbb{X})= \{0\}$.
\end{proof}

\begin{lemma}\label{lem:gl}
If  $A\in \gl{p}$ is invertible, then $L$ is $n$-estimable if and only if  $L_1=AL A\transm $ is $n$-estimable.
\end{lemma}
\begin{proof}
This follows as $L_1\cap \Sym{p}_0(\mathbb{X})= L\cap \Sym{p}_0(A\transm \mathbb{X})$.
\end{proof} 

We next identify $n$-estimability with the ability to transform $L$ into what we call \emph{standard form}.  This condition may be easier to check in some situations.
We first identify $K\in\Sym{p}$ with $A\in\Sym{n},B\in\mathbb R^{r\times n}$ and $C\in\Sym{r}$ via
\[
K = \begin{pmatrix}
A & B\transm \\ B & C
\end{pmatrix}.
\]
Denote by $\LowerCorner$ the subspace of $\Sym{p}$ with $C=0$. We then have $\dim\Sym{p}=T_p$, $\dim\Sym{r}=T_r$, and $\dim{\LowerCorner}=T_p-T_r$.
We say that $L$ has \emph{$n$-standard form} if 
\begin{equation}
\label{Lcondition}
L = \left\{\begin{pmatrix}
A & B\transm \\ B & F(A,B)
\end{pmatrix}: A\in\Sym{n}, B\in\mathbb R^{r\times n}, \right\}
\end{equation}
for some linear function $F: \Sym{n}\times \mathbb R^{r\times n}\to \Sym{r}$.
Note that if $L$ has $n$-standard form then we have $d=\dim L=\dim\LowerCorner= T_p-T_r.$
\begin{lemma}\label{lem:standard}
If $L$ has $n$-standard form then $L$ is $n$-estimable.
\end{lemma}
\begin{proof} If $L$ has $n$-standard form we can choose $x=(x^1,\ldots,x^n)$ as the first $n$ standard basis vectors $e^1,\ldots,e^n$ of $\mathbb R^p$. Then for any $K$ of the form \eqref{Lcondition} we have $Ke^i=0$ for all $i=1,\ldots, n$ if and only if $A=0$ and $B=0$.  Hence we must have $K=0$, so the quadratic form $D_2(K)$ is positive definite.
\end{proof}

\begin{corollary}\label{cor:estimability}If $L\subseteq L_0$ and for some $A\in \gl{p}$, $L_1=A L_0 A\transm$ has $n$-standard form, then $L$ is $n$-estimable.
\end{corollary}
\begin{proof}
This follows by combining Lemma~\ref{lem:subspace}, Lemma~\ref{lem:gl}, and Lemma~\ref{lem:standard}.
\end{proof}

We also note that the converse to Corollary~\ref{cor:estimability} holds, as shown in the following lemma.
\begin{lemma}\label{lem:estimability}If $L$ is $n$-estimable then there exists $A\in\gl{p}$ and $L_0 \supseteq L$ such that $L_1=A L_0 A\transm$ has $n$-standard form.
\end{lemma}
\begin{proof} Let $n$ be the smallest integer $m$ such that $L$ is $m$-estimable. 
Then there exist orthogonal vectors $x^1\ldots,x^n$ such that $W: K\to K\jordan \sum_{1}^nx^i{x^i}\transm/n$ has full rank. Let $A$ denote the transformation to a coordinate system where $x^1,\ldots,x^n$ are the first $n$ basis vectors. In this coordinate system we have \[W(K)= K\jordan \sum_{1}^nx^i{x^i}\transm/n =\begin{pmatrix}K_{11}&K_{12}/2\\K_{21}/2&0\end{pmatrix}/n\]
where 
\[ K =\begin{pmatrix}K_{11}&K_{12}\\K_{21}&K_{22}\end{pmatrix},\] with $K_{11}\in \Sym{n}, K_{12}\in \mathbb R^{n\times r}$ and $K_{22}\in \Sym{r}$. Since $L$ is $n$-estimable, this map is injective and thus we must have $K_{22}=F(K_{11}, K_{12})$ for some linear function $F$.  Thus in this basis $L\subseteq L_0$ where $L_0$ has standard form.
\end{proof}

Finally we show that all non-trivial $L$ with $d\leq T_p-T_r$ and $r=p-n=1$ are $n$-estimable.
\begin{proposition}
\label{prop:np1}
Suppose that $L$ is a linear subspace of $\Sym{p}$ with $L\cap \Sym{p}_+\neq \emptyset$ and that $\dim L \leq T_p-T_r$ with $r=1$.  Then $L$ is $n$-estimable.
\end{proposition}
\begin{proof}
For contradiction, assume $d=\dim L=T_{p}-1$ %$I\in L$, and
 and that $L$ is not $n$-estimable for $n=p-1$. Thus for \emph{any} $x^1,\ldots,x^n$, the map \[W: K\to K\jordan \sum_{1}^nx^i{x^i}\transm/n\] has rank less than $d$. Assume now that $x^i =e^i, i=1,\ldots, n$ are orthonormal so that $e^1,\ldots,e^p$ form an orthonormal basis for $\mathbb R^p$.
In this basis we have
\[\sum_{1}^ne^i{e^i}\transm=\begin{pmatrix} I_{n}& 0\\0&0\end{pmatrix}.\]
and thus \[W(K)= K\jordan \sum_{1}^ne^i{e^i}\transm/n =\begin{pmatrix}K_{11}&k/2\\k\transm/2&0\end{pmatrix}/n\]
where 
\[ K =\begin{pmatrix}K_{11}&k\\k\transm&k_{22}\end{pmatrix},\] with $K_{11}\in \Sym{n}, k\in \mathbb R^{p-1}$ and $k_{22}\in \mathbb R$. 

Since $L$ is not $n$-estimable, there must be an $A\in L$ with $W(A)=0$, which implies that $A_{ij}=0$ unless $i=j=p$. We may thus assume that 
\[A=\begin{pmatrix} 0_n& 0\\0&1\end{pmatrix}.\]
In the original basis, we have $A=e^i{e^i}\transm$. Since $e^1,\ldots,e^p$ were arbitrary, we have shown that for any vector $u$ of length one, $uu\transm \in L$. Since $L$ is a linear subspace we conclude that any matrix $K$ of the form
\[K=\sum_{i=1}^p \lambda_ie^i{e^i}\transm\] is in $L$, and hence $\Sym{p}\subseteq L$.  This implies $d=T_p$, which is a contradiction. We conclude that $L$ is $n$-estimable.
\end{proof}

Notice that even when $L$ is $n$-estimable, the estimated concentration matrix may not be positive definite if $L$ is not a Jordan subalgebra. All we can say is that the estimate will be positive definite for sufficiently large $n$ by the consistency result in  Proposition~\ref{prop:consistency}.

If $\check K$ is not positive definite and the estimate of $K$ itself is of interest, it may be necessary to calculate the MLE $\hat K$: the latter exists and is positive definite whenever $\check K$ exists. In any case, lack of positive definiteness of $\check K$ indicates that the estimate may not be reliable and that there could be too few observations to justify the use of a model of such complexity.

\section{Gaussian graphical models with symmetries}
Gaussian graphical models with symmetries \citep{Hojsgaard} are linear concentration models generated by a coloured graph.  More precisely, we let $\mathcal{G}=(\NSS,\ESS)$ denote a \emph{coloured graph} where $\NSS$ is a partition of a finite vertex set $V$ into \emph{vertex colour classes} and $\ESS$ a partition of an edge set $E$ into \emph{edge colour classes}.  Such a graph determines a linear concentration model with $L= \mathcal{S}(\NSS,\ESS)$ being the set of symmetric $p\times p$ matrices $K$ with entries $k_{\alpha\beta}=0$ whenever $\alpha$ and $\beta$ are not neighbours in $\mathcal{G}$, any two off-diagonal elements being identical if the corresponding edges are in the same colour class, and any two diagonal elements identical if the corresponding vertices are in the same colour class.

Let $e^{u}$ for $u \in \NSS$ denote the $|V| \times |V|$ diagonal matrix with $e^{u}_{\alpha\alpha} = 1$ if $\alpha \in u$ and 0 otherwise. Similarly, for each edge colour class $u \in \mathcal{E}$ we let $e^{u}$ be the $|V| \times |V|$ symmetric matrix with $e^{u}_{\alpha\beta} = 1$ if $\{\alpha,\beta\} \in u$ and 0 otherwise. Then $\{e^u, u\in \NSS\cup \ESS\}$ form an orthogonal basis for $L$. The likelihood equations \citep{Hojsgaard} become 
\begin{equation}\label{eq:mleggm} \tr(e^u W)
=  \tr(e^u K^{-1}), \quad u\in \NSS\cup \ESS,\end{equation}
which are non-linear in $K$ and must be solved by iterative methods in most cases. 

One motivation for introducing these models was the potential reduction in the number of parameters of the corresponding uncoloured graphical model.  This increases the stability of estimates and allows estimators to exist for a smaller number of observations.
The last issue was considered in detail by \citet{uhler} for specific examples. As an aside we note  that the models determined by the coloured graphs  11, 14 and 17 in \citep[Table 2]{uhler} are supermodels of the Jordan linear concentration model \eqref{eq:counterex} and hence we can confirm \citeauthor{uhler}'s conjecture that in these cases, the MLE does not exist for $n=1$. 

We note that our Proposition~\ref{prop:notEstimable} implies that the SME does not exist if 
$|\NSS|+|\ESS| > n(2|V|-n+1)/2$ and believe that the condition $|\NSS|+|\ESS| \leq n(2|V|-n+1)/2$ is sufficient to ensure $n$-estimability for this particular class of linear concentration models, but have not been able to show this except for the case of $n=p-1$, cf.\ Proposition~\ref{prop:np1}. However, none of the examples in \citep{jensen:88} or \citep{uhler} provide counterexamples to this conjecture. Note that the MLE may well exist even if the SME does not exist; see for example the earlier discussion of the four-cycle. If our conjecture is correct, Proposition~\ref{prop:sme2mle} implies that $|\NSS|+|\ESS| \leq n(2|V|-n+1)/2$ is also sufficient for the existence of the MLE and hence provides a simple method of checking for this.

The linear score matching equations \eqref{eq:score} for graphical Gaussian models with symmetries are
\begin{equation}\label{eq:smeggm}
 \tr(e^u W K)= \tr(e^u)% = I(u\in\NSS)|\NSS|
, \quad u\in \NSS\cup \ESS,
\end{equation}
which should be compared to \eqref{eq:mleggm}; they have a strong similarity with the Yule--Walker equations for estimating the parameters of autoregressive processes in a time series, as also noted by \citet{gidas}.

 Indeed, a circular autoregressive process of order $q$ is an example of a coloured graphical model with symmetry determined by the cyclic permutation group, as displayed in
  Figure~\ref{fig:circularAR}.  In this case the Yule--Walker equations are exactly equivalent to the score matching equations.  
\begin{figure}[htb]
\center
\begin{tikzpicture}
[node/.style={circle,fill=s1,inner sep=0pt,minimum size=5mm},
edge1/.style={draw=s2,line width=2pt},
edge2/.style={draw=s3,line width=2pt,loosely dashed}]
\node[node] (1)  at (canvas polar cs:angle=90,radius=2cm) {};
\node[node] (2)  at (canvas polar cs:angle=141,radius=2cm) {};
\node[node] (3)  at (canvas polar cs:angle=192,radius=2cm) {};
\node[node] (4)  at (canvas polar cs:angle=243,radius=2cm) {};
\node[node] (5)  at (canvas polar cs:angle=294,radius=2cm) {};
\node[node] (6)  at (canvas polar cs:angle=345,radius=2cm) {};
\node[node] (7)  at (canvas polar cs:angle=39,radius=2cm) {};
\draw[edge1] (1)--(2);
\draw[edge1] (2)--(3);
\draw[edge1] (3)--(4);
\draw[edge1] (4)--(5);
\draw[edge1] (5)--(6);
\draw[edge1] (6)--(7);
\draw[edge1] (7)--(1);
\draw[edge2] (1)--(3);
\draw[edge2] (2)--(4);
\draw[edge2] (3)--(5);
\draw[edge2] (4)--(6);
\draw[edge2] (5)--(7);
\draw[edge2] (6)--(1);
\draw[edge2] (7)--(2);
\end{tikzpicture}
\caption{The circular autoregressive process of order 2 with $n=7$ as a coloured Gaussian graphical model.}
\label{fig:circularAR}
\end{figure}
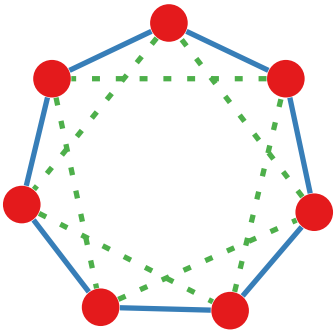

\subsection{Model selection}
\label{sec:screen}
Model selection in Gaussian graphical models with symmetry is problematic as the number of possible models is enormous. This affects both stepwise methods, as used in \citep{gRc}, and lattice based methods \citep{edwards:87}, as used in \citep{gehrmann:11}. 
The computational efficiency of the SME allows rapid screening of a large number of potential models as the minimized objective function indicating the model fit $J_2(\check\theta)=-n\tr(\check K)/2$ is particularly simple to calculate. Note that this minimum can be calculated even though $\check K$ may not be positive definite; in particular a time consuming check of positive definiteness can then be avoided. In this case the minimum may overestimate the model fit as it corresponds to the minimum of $J_2$ over the entire space $L$ rather than over $\Theta\subseteq L$.

To prevent overfitting a penalty for the number of parameters $d=\dim L$ should be added to $J_2$ to give the objective function
\begin{equation}\label{eq:penalty}J_\lambda(L)= J_2(\check K)/n+\lambda d=-\frac{1}{2}\tr(\check K) + \lambda d.\end{equation}  The scalar multiple $\lambda$ for the penalty can, for example, be determined by a method such as cross-validation. Such rapid model screening may be useful to identify a small set of plausible models to be considered by more sophisticated search procedures.

\subsection{Examples}
\label{sect:application}

We briefly describe some numerical experiments with the SME for Gaussian graphical models with and without symmetries.  These indicate that the SME provides an extremely fast estimate which is reasonably accurate for large $n$.  

\begin{figure}[htb]
\center
		\begin{tikzpicture}
		[node1/.style={circle,fill=s2,inner sep=0pt,minimum size=5mm},
		node2/.style={circle,fill=s1,inner sep=0pt,minimum size=5mm},
		node3/.style={circle,fill=s3,inner sep=0pt,minimum size=5mm},
		edge3/.style={draw=s4,line width=2pt},
		edge2/.style={draw=s5, line width=2pt,dash pattern =on 10pt off 7pt},
		edge1/.style={draw=s6,line width=2pt,dash pattern =on 3pt off 10pt}]
		\node[node1,label=below:Algebra] (Algebra) at (0,0) {}; %[label=above:Algebra]
		\node[node2,label=above:Vectors] (Vectors)  at (-4.5,1.5) {1};
		\node[node2,label=above:Analysis] (Analysis) at (4.5,1.5) {1};
		\node[node3,label=below:Mechanics] (Mechanics) at (-4.5,-1.5) {2};
		\node[node3,label=below:Statistics] (Statistics) at (4.5,-1.5) {2};
		\draw[edge1] (Mechanics)--(Vectors);
		\draw[edge2] (Analysis)--(Statistics);
		\draw[edge1] (Mechanics)--(Algebra);
		\draw[edge3] (Statistics)--(Algebra);
		\draw[edge3] (Vectors)--(Algebra);
		\draw[edge2] (Analysis)--(Algebra);
		\end{tikzpicture}
		\caption{A Gaussian graphical model with symmetries for the \textbf{mathmarks} dataset.}
		\label{fig:mathgraph}
		\end{figure}
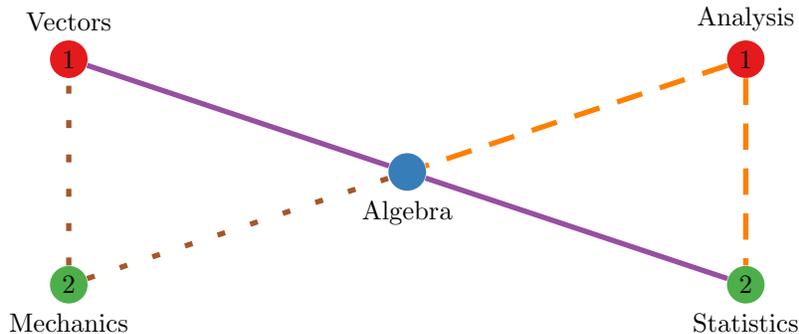
		
First we consider the \textbf{mathmarks} dataset \citep{mardia:kent:bibby:79} included in \textbf{gRc} \citep{gRc}.  Following the analysis in \citep{Hojsgaard}, we use three vertex colour classes and three edge colour classes as shown in Figure~\ref{fig:mathgraph}.  We vary the number of observations $n$ from 4  to 88  and compute both the SME and the MLE for each $n$.  Figure~\ref{fig:converge} shows how the SME approximates the MLE as $n$ grows.   The SME appears to provide a computationally efficient estimator with good accuracy for large $n$.

\begin{figure}[htb]
\center
\begin{tikzpicture}
\begin{axis}[width=11.5cm,height=6cm, xlabel=$n$, ylabel=$\sqrt{n}\sqrt{\tr\{ (\check K - \hat K)^2\}}$,yticklabel style={/pgf/number format/fixed,/pgf/number format/precision=3}, scaled ticks=false,xmin=0.9,xmax=91] 
\addplot+[thick,no markers,color=s2] table {ScaledFrobenius.txt};
\end{axis}
\end{tikzpicture}
	\caption{The scaled Frobenius distance between the SME and the MLE for the \textbf{mathmarks} database as $n$ increases.}
	\label{fig:converge}
\end{figure}
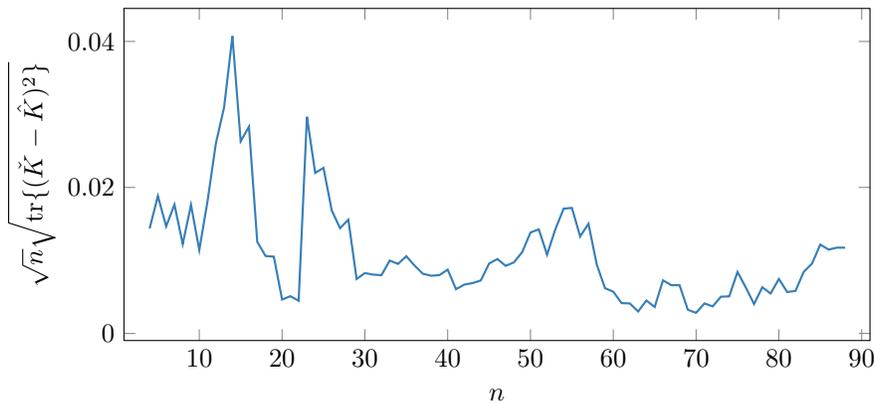

Next we shall give an example of the SME identifying a non-decomposable graphical Gaussian model.   We first simulated data from a square lattice model with $p=s^2$ vertices: the concentration matrix for our model is symmetric with upper-triangular entries
\[
K_{ij}=\begin{cases}
1\quad&\mbox{if } j=i\\
0.2\quad&\mbox{if } j=i+s\\
0.2\quad&\mbox{if } j=i+1 \mbox{ and } i\not\equiv 0 \imod{s}\\
0\quad&\mbox{otherwise}
\end{cases}
\]
for $1\leq i\leq j\leq s^2$.
The previously mentioned four-cycle (Figure~\ref{fig:fourcycle}) is such a model with $s=2$.  We then used the SME to conduct a rapid model search over uncoloured graphs $G$ on $p$ vertices.   

We now describe our model search method.  We began by initializing the graph to the best-fitting tree via Kruskal's algorithm \cite{kruskal} using squared correlations as weights.  This initialization is very fast, and if we were only searching over trees, it corresponds to the maximum likelihood estimate \citep{chow:liu:68,edwards:etal:10}.  In our case we are searching among all undirected graphs, but the tree step provides a computationally efficient starting position for the search.  Next we conducted a greedy ``forward search'' and successively add edges to $G$.  To ensure our method is scalable for large $p$, we considered adding edges in order of decreasing squared correlations and we terminated the forward search after attempting to add an edge that fails to improve the objective function.  Finally, we conducted a greedy ``backward search'' by successively removing existing edges from $G$.

Before running this algorithm we identified a suitable penalty $\lambda$.  We considered the change in the objective function $J_2$ after adding the single extra edge which most improved the objective to the true model.  By simulating several thousand samples over the grid with $s=2,\ldots,8$ and $n=s^2,2s^2,3s^2,\ldots,10s^2$, we found that the expected change in $J_2$ was approximately proportional to $\sqrt{p}\log\log(np)$.  We chose to proceed with $\lambda=\sqrt{p}\log\log(np)/(2n)$.

 We quantified the accuracy of the fitted model $\check K$ by considering the number of missing edges ($\check K_{ij}=0$ but the true $K_{ij}\neq 0$) and extra edges  ($\check K_{ij}\neq 0$ but the true $K_{ij}=0$) in Figure~\ref{fig:nMissingExtra}. 

\def\whichstep{Best}
\begin{figure}[htp]
\center
\begin{tikzpicture}
\pgfplotstableread[header=true]{NumExtraMissing.txt}\loadedtable
\begin{axis}[width=12cm,height=7cm, ymin=-10, ymax=340, xlabel=$n/ s^2$, ylabel=Number of missing and extra edges,scaled ticks=false,no markers,every axis plot/.append style={thick},xmin=0.9,xmax=10.1]
\addplot+[color=s1] table[x=nmp,y={nExtra16\whichstep}] {\loadedtable};
\addplot+[color=s2] table[x=nmp,y={nExtra64\whichstep}] {\loadedtable};
\addplot+[color=s3] table[x=nmp,y={nExtra256\whichstep}] {\loadedtable};
\legend{$s=4$,$s=8$,$s=16$}
\addplot+[dashed,color=s1] table[x=nmp,y={nMiss16\whichstep}] {\loadedtable};
\addplot+[dashed,color=s2] table[x=nmp,y={nMiss64\whichstep}] {\loadedtable} ;
\addplot+[dashed,color=s3] table[x=nmp,y={nMiss256\whichstep}] {\loadedtable};
\end{axis}
\end{tikzpicture}
\caption{The number of missing (dashed line) and extra (solid line) edges in the SME for the lattice with $s=4,8,16$ and $n$ from $n=s^2$ to $n=10s^2$.}
\label{fig:nMissingExtra}
\end{figure}
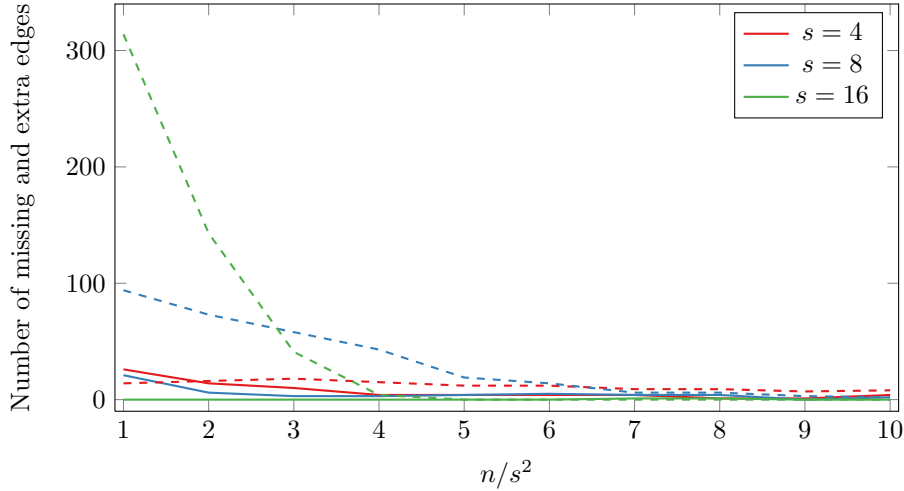

Finally, we visually inspect how the score matching estimate becomes more accurate with increasing $n$ in Figure~\ref{fig:sparsityimages}.   We see that the estimate is unstable at $n=p$, but for $n=10p$ the SME correctly identifies the majority of the true non-zero entries with few extra edges.  Note that for $p=256$ and $n=10p$ the model is correctly identified. 

\def\pletter{p}
\begin{figure}[htp]
\foreach \nS in {1,5,10}{
\foreach \p in {16,64,256}{
\pgfmathsetmacro\n{\p*\nS}
	\begin{subfigure}[b]{0.33\textwidth}
		\centering
 \frame{\includegraphics[width=\textwidth]{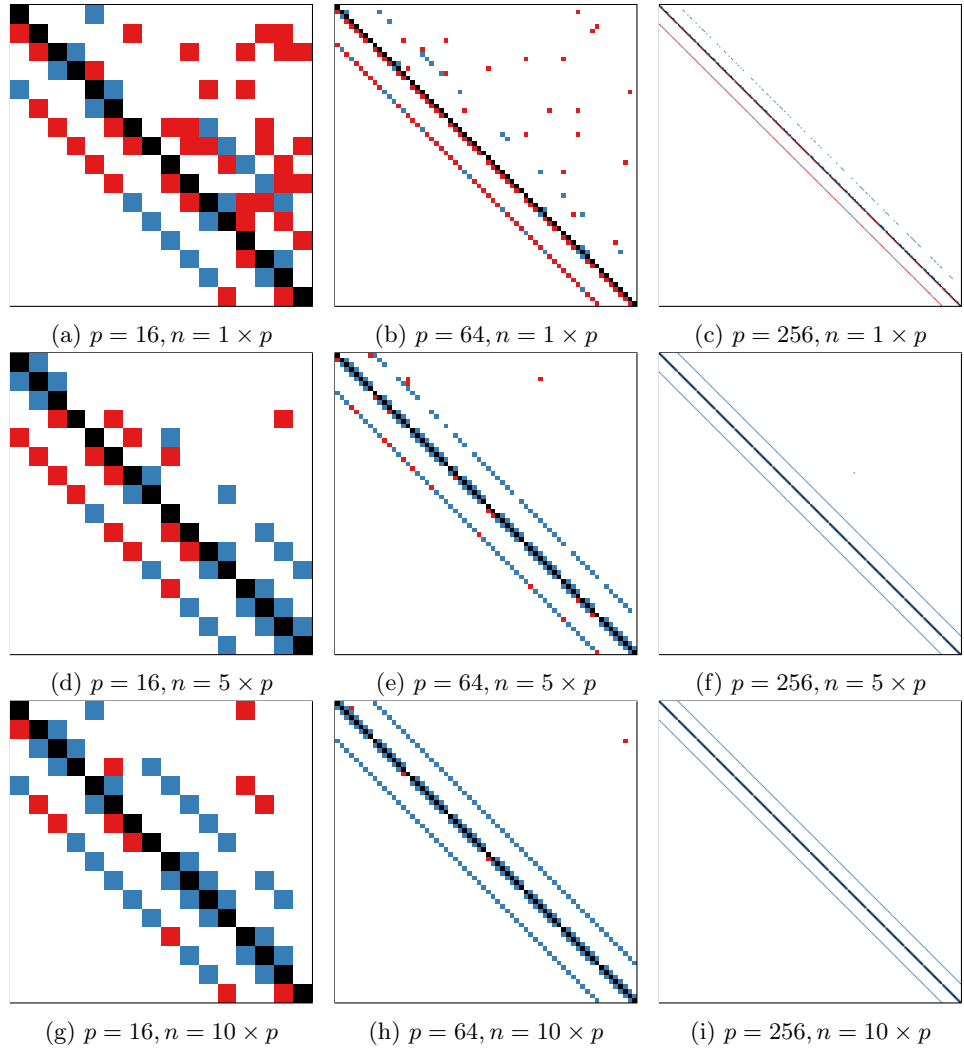}}
	\caption{$p=\p,n=\nS\times p$}
	\end{subfigure}~
}\\
}
\caption{Sparsity patterns of the estimated concentration matrices.  Above the diagonal is the SME $\check K$; below the diagonal is the true concentration matrix $K$.  Entries that are present in both $K$ and $\check K$ are blue and entries which are present in only one matrix are red.  Thus red entries below the diagonal correspond to true edges that were missed by the SME and red entries above the diagonal correspond to extra edges in the SME.}
\label{fig:sparsityimages}
\end{figure}

In the case of searching for uncoloured graphical models as above, the SME may be seen as alternative to the graphical lasso algorithm \cite{friedman:etal:08}. It is difficult to directly compare the accuracy of the SME to the graphical lasso due to the unspecified regularization parameter in the latter algorithm.  The graphical lasso estimate is extremely sensitive to the precise value of this regularization parameter: by fine-tuning the parameter for each $n$ and $p$ we were able to achieve results equal to or better than those of the SME, however the range of values which yields accurate estimates is narrow and highly dependent on $n$ and $p$.  By contrast the SME seems relatively robust against small changes in $\lambda$.

We also suspect that the SME scales better for large $p$ than the graphical lasso.  Using a implementation of the SME written in $C\#$ we were able to consider models up to $s=100$ ($p=10^4$) before running out of computer memory: even at this large $p$ each estimate of the SME could be completed within ten seconds, and with $n=10p$ our rudimentary search procedure correctly identified the model.  We attempted to test the graphical lasso at such $p$ using the R package \textbf{glasso} \citep{glasso}, however the R environment ran out of memory while attempting to load the sample covariance matrix.

Finally, we should emphasize that, in contrast to the graphical lasso, the SME can be used for graphical model with symmetries.  Thus the SME may be useful for model screening over such models, though this would require the development of computationally efficient model search strategies. It is outside the scope of the present article to study such strategies in any detail.

\section{Discussion}
Score matching is an efficient method of parameter estimation for distributions with intractable normalization constants. It is particularly suitable to parameter estimation within an exponential family, where the score estimating equations are linear and yield a consistent estimate.  The ready availability of highly optimized algorithms for linear equations means that the SME can be computed quickly and with a small memory footprint, even when the number of parameters is very large.

The method seems particularly promising for rapid model screening and it would be well worthwhile to further investigate the optimal form of the penalty for model complexity, i.e.\ the coefficient $\lambda$ in \eqref{eq:penalty}, in particular how it should depend on $n$ to ensure consistent identification of the model  along the lines in \cite{hannan:quinn:79}, see also the discussion in \cite{Erven}.
It would be an advantage to have a simple sufficient condition for the existence of the SME --- and hence also the MLE --- in Gaussian graphical models (with or without symmetries) so that a model search could be automatically restricted to models of sufficiently limited complexity. We believe the condition in Proposition~\ref{prop:notEstimable} is necessary and sufficient for the SME to exist in this case and thus sufficient for existence of the MLE. Unfortunately we only been able to show this in general for $n\geq p- 1$. 
We hope to return to these and other questions in the future.

%\section*{References}
%\bibliographystyle{chicago}
%\bibliography{ScoreMatching}

\end{document}